\documentclass{amsart}
\usepackage[T1]{fontenc}
\usepackage[latin1]{inputenc}
\usepackage{color}
\usepackage{graphicx}
\usepackage{subfig}
\usepackage{syntonly}

\usepackage{float}

\usepackage{amssymb,amsmath,amsfonts, amsthm}

\usepackage{mathrsfs}

\usepackage{tabularx}

\usepackage{hyperref}

\begin{document}
	\setlength{\parindent}{0pt}

	\author{Malte Lackmann}
	\address{Mathematisches Institut\\
		Universit\"at Bonn}
	\email{lackmann@math.uni-bonn.de}
	
	\title{The octonionic projective plane}
	\maketitle
	
	\begin{abstract} This small note, without claim of originality, constructs the projective plane over the octonionic numbers and recalls how this can be used to rule out the existence of higher-dimensional real division algebras, using Adams' solution of the Hopf invariant $1$ problem.
	\end{abstract}

	
	
	\setlength{\parindent}{0pt}
	\newcommand{\absatz}{\\[0.5cm]}
	\newcommand{\kleinerabsatz}{\\[0.25cm]}
	\newcommand{\bw}[1]{\textcolor{red}{#1}}

	\newcommand{\oc}{$\mathbb{O}$}
	\newcommand{\R}{$\mathbb{R}$}
	\newcommand{\C}{$\mathbb{C}$}
	\newcommand{\ha}{$\mathbb{H}$}
	\newcommand{\hide}[1]{}

	\newtheorem{thm}{Theorem}[section]
	
	\newtheorem{lemma}[thm]{Lemma}
	
	\newtheorem{prop}[thm]{Proposition}
	
	\newtheorem*{notation}{Notation}
	
	\newtheorem*{definition}{Definition}
	
	\newtheorem{corollary}[thm]{Corollary}
	
	\theoremstyle{remark}
	\newtheorem{remark}{Remark}
	\newtheorem*{question}{Open Question}
	
	\newcommand{\pr}{{\bf Proof: }}
	
	\date{}
	\maketitle
	
	\section{Introduction}
	
	As mathematicians found out in the last century, there are only four normed division algebras\footnote{A division algebra over \R\, is a finite-dimensional unital \R-algebra without zero divisors, not necessarily commutative or associative. A normed algebra over \R\, is an \R-algebra $\mathbb{A}$ together with a map $\left\Vert\cdot\right\Vert\colon \mathbb{A}\rightarrow \mathbb{R}$ coinciding with the usual absolute value on $\mathbb{R}\cdot 1\cong \mathbb{R}$, satisfying the triangular inequality, positive definiteness and the rule $\left\Vert xy\right\Vert=\left\Vert x\right\Vert\left\Vert y\right\Vert$.} over $\mathbb{R}$: the real numbers themselves, the complex numbers, the quaternions and the octonions. Whereas the real and complex numbers are very well-known and most of their properties carry over to the quaternions (apart from the fact that these are not commutative), the octonions are very different and harder to handle since they are not even associative. However, they can be used for several interesting topological constructions, often paralleling constructions known for \R, \C\, or \ha. In this article, we will construct $\mathbb{O}\textup{P}^2$, a space having very similar properties to the well-known two-dimensional projective spaces over \R, \C\,and \ha. 
	
	\subsection*{Outline} The second section will recall a construction of the octonions and discuss their basic algebraic properties.
	We will then move on to actually construct the octonionic projective plane $\mathbb{O}\textup{P}^2$. In the fourth section, we will discuss properties of $\mathbb{O}\textup{P}^2$ and applications in algebraic topology. In particular, we will use it to construct a map $\mbox{S}^{15}\longrightarrow \mbox{S}^8$ of Hopf invariant $1$. Moreover, it will be explained why there cannot be projective spaces over the octonions in dimensions higher than $2$. 
	
	\subsection*{Acknowledgements} This note was originally written by the author as a term paper for the Algebraic Topology 1 course of Peter Teichner at the University of Bonn. I thank him for suggesting this interesting topic. At the time of publication, I am working as a PhD student at Bonn, supported by the ERC Advanced Grant "KL2MG-interactions" (no. 662400) of Wolfgang L\"uck. I thank Jasper Knyphausen, Stefan Friedl and the anonymous referee for helpful comments on this paper.
	
	\section{Construction of \oc}
	
	\label{construction}
	
	In this section, we will explain how the octonions can be constructed. Of course, there are many ways to define them -- the simplest way would be to choose a basis of \oc\, as a vector space and then specify the products of each pair of basis elements. We will try to explain a little better the ``reason'' for the existence of octonions by giving a construction which leads from the quaternions to the octonions, but which can also be used to construct \C\, out of \R\, and \ha\, out of \C, the so-called \emph{Cayley-Dickson construction}.

	As mentioned in the introduction, the octonions are neither commutative nor even associative. However, they have the following crucial property called \emph{alternativity}:
	
	\begin{quotation}
		For any two octonions $x$ and $y$, the subalgebra of \oc\,  generated by $x$ and~$y$ is associative.
	\end{quotation}
	
	Here a subalgebra is always meant to contain the unit. By a nontrivial theorem of Emil Artin \cite{zorn, schafer}, the above condition is equivalent to requiring that the formulas 
	\begin{align}\label{eq:alt} x(yy)=(xy)y\quad \textup{and}\quad (xx)y=x(xy)\end{align} hold for any two elements $x,y\in\mathbb{O}$.

	We will now begin to construct the octonions. To do this, note that the well-known division algebras \R, \C\, and \ha\, are not only normed \R-algebras, but they come together with a \emph{conjugation}: an anti-involution $^{*}$ (i.e., a linear map $^*$ from the algebra to itself such that $(xy)^* = y^*x^*$ and $(x^*)^*=x$) with the property that $xx^*=x^*x=\left\Vert x \right\Vert^2$.
	
	This extra structure goes into the following construction, called ``doubling construction'' or ``Cayley-Dickson construction'': Let $(\mathbb{A}, \left\Vert\cdot\right\Vert, ^*)$ be a normed real algebra with conjugation. Then we define the structure of a normed real algebra with conjugation on the real vector space $\mathbb{A}^2$ by
	\begin{itemize}
		\item $(a,b) \cdot (c,d) = (ac-d^*b, da+bc^*)$, 
		\item $\left\Vert(a,b)\right\Vert = \sqrt{\left\Vert a\right\Vert^2+\left\Vert b\right\Vert^2}$,
		\item $(a,b)^* = (a^*,-b)$.
	\end{itemize}
	
	It can be checked that this turns $\mathbb{A}^2$ indeed into a real algebra with conjugation. Furthermore, it also conserves most of the properties as an algebra that $\mathbb{A}$ has had, though not all: 
	\begin{itemize}
		\item The real numbers are an associative and commutative division algebra and have the additional property that the conjugation is just the identity. Applying the Cayley-Dickson construction, we obtain the complex numbers which are still an associative and commutative division algebra, but the conjugation is not trivial any longer -- it is the usual complex conjugation. 
		\item Going from \C\, to \ha, we lose the property of commutativity: \ha\, is only an associative division algebra. 
		\item Applying the construction to \ha, we get the octonions \oc, which are not even associative any more, but are still a normed division algebra and are alternative as explained above.
		\item Continuing to apply the Cayley-Dickson construction, one obtains a $16$-dimensional real algebra called the \emph{sedenions}. The sedenions have zero divisors, and therefore cannot have a multiplicative norm. They are by far less important than the other four division algebras. However, they still satisfy a property called \emph{flexibility} which is a very weak form of associativity. Interestingly, if the Cayley-Dickson construction is applied again, this property is \emph{not} lost, so there are flexible $2^n$-dimensional normed real algebras for every $n$ \cite{moreno}.
	\end{itemize} 
	
	The proofs of these statements are rather lengthy, but straightforward. We will carry them out for the passage from quaternions to octonions, since this is the case we are most interested in and also the most difficult one, and leave the remaining cases to the reader.
	
	\begin{prop}
		\oc\, is an alternative normed division algebra.
	\end{prop}
	
	\begin{proof}
		For the alternativity, we prove \eqref{eq:alt} and then use Artin's theorem to deduce alternativity. We write $x=(a,b)$ and $y=(c,d)$ with $a, b, c, d\in \mathbb{H}$ and write out both sides of the first equation, using the definition of the multiplication displayed above. Doing the calculation and obvious cancellations, this leaves us to show the two identities
		\[d^*bc+d^*da+d^*bc^* = add^* + c^*d^*b+cd^*b\]
		and 
		\[dac -dd^*b+dac^* = dca+dc^*a-bdd^*\,.\]
		Considering the first equation, note that $dd^*=d^*d=\left\Vert d \right\Vert^2$ is a real number and thus central, so that $d^*da=add^*$. The remaining terms can be regrouped as follows:
		\[d^*b(c+c^*)=(c+c^*)d^*b\,.\]
		However, $c+c^*$ is a real number as well, as can for instance be seen easily from the Cayley-Dickson construction, so that we have proved the first of the two identities. The second can be traced back similarly to the facts that $dd^*b=bdd^*$ and
		\[da(c+c^*)=d(c+c^*)a\,.\]
		
		We now prove that the octonion norm is multiplicative. This is also not completely formal, as can be seen from the fact that it is not true for the sedenions. With notation as above, the equation $\left\Vert(a,b)\right\Vert^2\left\Vert(c,d)\right\Vert^2=\left\Vert(a,b)(c,d)\right\Vert^2$ can be simplified to 
		\begin{align} \label{norm} acb^*d + d^*bc^*a^* = bacb^* + bc^*a^*d^*\,.\end{align}
		Following \cite[p.~48]{kantor}, we consider two cases: If $d$ is real, the equation holds true trivially. If $d$ is purely imaginary in the sense that $d^*=-d$, then the equation is equivalent to
		\[d(acb^*+bc^*a^*) = (acb^*+bc^*a^*)d\,,\]
		which is true since
		\[acb^* + bc^*a^* = acb^* + (acb^*)^*\]
		is real. By linearity, \eqref{norm} is true for all $d$.
		
		The multiplicativity of the norm at hand directly implies the fact that \oc\ is a division algebra: If $xy=0$, then 
		\[\left\Vert x\right\Vert \left\Vert y\right\Vert = \left\Vert xy\right\Vert=0\,,\]
		so $\left\Vert x\right\Vert = 0$ or $\left\Vert y\right\Vert=0$ and thus $x=0$ or $y=0$ since $\left\Vert \cdot\right\Vert$ is a norm.
	\end{proof} 
	
	\begin{remark}
		(i) Note that all formulas of the above proof are written without parantheses, thus we used  secretly that \ha\ is associative. This is essential: As mentioned above, the sedenions, constructed out of the non-associative octonions, are neither alternative, nor a division algebra (and, consequently, they cannot possess a multiplicative norm).

		(ii) The book \cite{conway-smith} gives a geometric argument that the octonions are alternative, which does not use Artin's theorem. See Section~6.8, in particular Theorem~2.
	\end{remark}
	
	\section{Construction of $\mathbb{O}\textup{P}^2$}
	
	This section discusses a construction of a projective space of dimension  $2$ over the octonions (which is very similar to the construction in \cite{baez}). Our goal is to get spaces with similar properties as their analogues over the real and complex numbers and the quaternions. The na\"{i}ve ansatz would be to define the $n$-dimensional octonionic projective space as a quotient of $\mathbb{O}^{n+1}\setminus \{0\}$, identifying every vector with its (octonionic) multiples. However, when one starts calculating, one sees that associativity is needed for this to be an equivalence relation. Since the octonions are not associative, we have to be more careful. In fact, the construction given in the following only works for $n\le 2$ (thanks to the property of alternativity), and we will see that there are theoretical obstructions to the existence of higher $\mathbb{O}\textup{P}^n$.
	
	One main difference to the other three projective planes is that we don't construct $\mathbb{O}\textup{P}^2$ as a quotient of $\mathbb{O}^3\setminus \{0\}$, but we restrict ourselves to the subset
	\begin{align*}
	T=\left\{(x,y,z)\in \mathbb{O}^3;\; \left\Vert x\right\Vert^2 + \left\Vert y\right\Vert^2 + \left\Vert z\right\Vert^2=1 \quad \textup{and}\right. & \textup{the subalgebra generated } \\ & \left. \textup{by $x, y$ and $z$ is associative}\vphantom{\left\Vert x\right\Vert^2} \right\}\,.
	\end{align*}
	
	We call two triples  $(x,y,z), (\tilde{x}, \tilde{y}, \tilde{z})\in T$ equivalent, $(x,y,z) \sim (\tilde{x}, \tilde{y}, \tilde{z})$, if and only if the six equations \[ xx^* = \tilde{x}\tilde{x}^*, \quad xy^* = \tilde{x}\tilde{y}^*, \quad xz^* = \tilde{x}\tilde{z}^*, \quad yy^* = \tilde{y}\tilde{y}^*, \quad yz^* = \tilde{y}\tilde{z}^*, \quad zz^* = \tilde{z}\tilde{z}^* \] hold. (The three remaining relations of a similar form follow from these by the properties of the conjugation.)  It is obvious that this is an equivalence relation, and thus we can set \[\mathbb{O}\mbox  {P}^2 = T/\sim \,.\]
	
	\subsection{Manifold structure} A first observation is that our space $\mathbb{O}\textup{P}^2$ constructed like this is quasi-compact since it is the quotient of a quasi-compact space. We will now show that it's a $16$-dimensional real manifold.
	
	The proof uses the following construction. Let $a, b, c$ be real numbers. Define an \R-linear map \[\ell=\ell_{(a,b,c)}\colon \mathbb{O}^3\longrightarrow \mathbb{O}\,, \quad \ell(x, y, z) = ax+by+cz\,.\] Note that for two triples $(x,y,z) \sim (\tilde{x}, \tilde{y}, \tilde{z})$ in $T$, we have $\ell(x,y,z)=0$ if and only if $\ell(\tilde{x}, \tilde{y},\tilde{z})=0$, since $\ell(x,y,z)$ vanishes exactly if $\ell(x,y,z) \ell(x,y,z)^*$ vanishes, and $\ell(x,y,z) \ell(x,y,z)^* = \ell(\tilde{x},\tilde{y},\tilde{z}) \ell(\tilde{x},\tilde{y},\tilde{z})^*$ by the definition of $\sim$ and since $a, b, c$ are real.
	
	Thus we get a well-defined open set \[U_\ell = U_{(a,b,c)} = \{[x,y,z];\, \ell(x,y,z) \neq 0\} \subset \mathbb{O}\textup{P}^2.\]
	
	\begin{prop}
		$\mathbb{O}\textup{P}^2$ is a locally Euclidean topological space.
	\end{prop}
	
	\begin{proof}
		Suppose that $c\neq 0$. Consider the maps 
		\begin{align*}
		\varphi_\ell: U_\ell \rightarrow \mathbb{O}^2\,, \quad
		[x,y,z] \mapsto \left(\frac{x\ell^*}{\left\Vert \ell \right\Vert^2}, \frac{y\ell^*}{\left\Vert \ell \right\Vert^2}\right)\,,
		\end{align*}
		where $ \ell = \ell(x,y,z)$. This map is well-defined (the argument given above shows that $\left\Vert\ell\right\Vert^2=\ell\ell^*$ only depends on $[x,y,z]$ and not on $(x,y,z)$; the same argument works for $x\ell^*$ and $y\ell^*$) and thus continuous by the universal property of the quotient topology.
		An inverse map is given by
		\begin{align*}
		&\psi_\ell: \mathbb{O}^2 \rightarrow  U_\ell\,, \\
		(x,y) \mapsto \left[\frac{x}{r}, \frac{y}{r}, \frac{1-ax-by}{cr}\right], &\quad r = \sqrt{\left\Vert x\right\Vert^2 + \left\Vert y\right\Vert^2 +\frac{1}{c^2} \left\Vert 1-ax-by\right\Vert^2}\,. 
		\end{align*}
		Recall that in the definition of $\mathbb{O}\textup{P}^2$, we only consider triples whose entries lie in an associative subalgebra of $\mathbb{O}$. Thus this map is only well-defined since $\mathbb{O}$ is alternative, as defined in Section \ref{construction}. Here we use that for every octonion $y$, the conjugate $y^*$ lies in the subalgebra generated by $y$ since it only differs from $-y$ by a real number. This can directly be seen from the Cayley-Dickson construction.
		
		Note that equation \eqref{eq:alt} would not be sufficient at this point, but we use the nontrivial Artin theorem which states that \eqref{eq:alt} implies alternativity. Moreover, it is exactly at this point (and in the next paragraph) where our procedure breaks down if we want to construct higher $\mathbb{O}\textup{P}^n$ in the same way.

		Checking that $\varphi_\ell$ and $\psi_\ell$ are inverse to each other is a simple calculation. The reader may amuse herself by reproducing them. Note that we evade associativity of $\mathbb{O}$ by carrying out calculations only with elements lying in an associative subalgebra of $\mathbb{O}$ -- by definition of $T$ for the one and by the alternativity of $\mathbb{O}$ for the other direction.
		
		The same obviously works with $a$ or $b$ instead of $c$, thus for all triples $(a,b,c)\neq (0,0,0)$. Thus we have covered $\mathbb{O}\textup{P}^2$ by the three chart regions $U_{(1,0,0)}, U_{(0,1,0)}$ and $U_{(0,0,1)}$.
	\end{proof}
	
	\begin{remark}
		(i) It is easy to check that the coordinate changes are smooth, such that $\mathbb{O}\textup{P}^2$ in fact becomes a smooth manifold.
		
		(ii) The referee raised the question whether $\mathbb{O}\textup{P}^2$ admits the structure of a complex manifold. Since it is an $(n-1)$-connected $2n$-manifold for $n=8$, we may apply the criterion of \cite[Thm.~1]{yang}, which says that $\mathbb{O}\textup{P}^2$ doesn't even admit a stable almost complex structure. Note that $\mathbb{H}\textup{P}^2$ admits a stable almost complex structure, but no almost complex structure \cite[Thm.~1, 2]{yang}.
	\end{remark}
	
	\begin{lemma}
		$\mathbb{O}\textup{P}^2$ is Hausdorff.
	\end{lemma}
	
	\begin{proof} Let $(x,y,z)$ and $(x', y', z')$ two elements in $T$. We have to find $(a,b,c)\in \mathbb{R}^3$ such that $\ell_{(a,b,c)}(x,y,z) \neq 0$ and $\ell_{(a,b,c)}(x',y',z')\neq 0$, since then it follows that $[x,y,z]$ and $[x',y',z']$ both lie in the open subset $U_{(a,b,c)}$ which we already know  to be Hausdorff.
		
		To find $(a,b,c)$ as above, note that given $x$, $y$ and $z$, the set of all solutions to the equation $ax+by+cz=0$ is a subspace of $\mathbb{R}^3$ of dimension at most $2$, and the same is true for the equation $ax'+by'+cz'=0$. Since the union of two planes is never the whole $\mathbb{R}^3$, we find a point that fails to satisfy both of these equations. 
	\end{proof}
	
	\begin{corollary}
		$\mathbb{O}\textup{P}^2$ is a closed $16$-dimensional real manifold.\qed
	\end{corollary}
	
	\subsection{The projective line $\mathbb{O}\textup{P}^1$}

	Consider the closed subset of $\mathbb{O}\textup{P}^2$ given by all equivalence classes of the form $[x,y,0]$ with $x, y\in \mathbb{O}$. It is immediate from the definitions that this is homeomorphic to the space \[\{(x,y) \in \mathbb{O}^2; \left\Vert x\right\Vert^2 + \left\Vert y\right\Vert^2 =1\}/\sim\,,\] where $(x,y)\sim (\tilde{x}, \tilde{y})$ if and and only if the three relations \[xx^* = \tilde{x}\tilde{x}^*, xy^* = \tilde{x}\tilde{y}^*, yy^* = \tilde{y}\tilde{y}^*\]
	hold. This is of course the analogue of our construction of $\mathbb{O}\textup{P}^2$ in one dimension lower, so we call the resulting space $\mathbb{O}\textup{P}^1$. 
	
	By the same argumentation as in the previous paragraph, $\mathbb{O}\textup{P}^1$ is a closed manifold. Moreover, there is a homeomorphism $\mathbb{O}\textup{P}^1 \cong \mbox{S}^8$. This can be constructed in the very same way as in the familiar cases over \R, \C\, or \ha\, (since the map in one direction involves only one element of \oc\, at a time).
	
	\subsection{CW structure}
	
	The octonionic projective plane has a very simple cell structure with one cell in each of the dimensions $0$, $8$ and $16$. This cell structure is constructed in exactly the same way as the analogous structures for $\mathbb{R}\textup{P}^2$, $\mathbb{C}\textup{P}^2$ and $\mathbb{H}\textup{P}^2$.
	
	A little lemma that we will need in the proof of the following statements is the observation that we can choose a vector space isomorphism $\mathbb{O}\cong \mathbb{R}^8$ such that the norm $\left\Vert\cdot\right\Vert$ becomes the usual Euclidean norm on $\mathbb{R}^8$. This follows directly from the definition of \oc\ via the Cayley-Dickson construction, but it can also be deduced from the formal properties of the conjugation map: $\langle x,y\rangle = \frac{1}{2}(x^*y+y^*x)$ defines a symmetric, positive definite bilinear form on the vector space $\mathbb{O}$, so we can find an orthonormal basis by Sylvester's law of inertia. 
	
	To begin with, consider the canonical inclusion \[\mathbb{O}\textup{P}^1 \hookrightarrow  \mathbb{O}\textup{P}^2, \quad [x,y] \mapsto [x,y,0]\,.\]
	By our definition of $\mathbb{O}\textup{P}^1$, this map is a homeomorphism onto its image, which is closed in $\mathbb{O}\textup{P}^2$. Since  $\mathbb{O}\textup{P}^1 \cong \mbox{S}^8$, we can use  $\mathbb{O}\textup{P}^1$ as the $8$-skeleton in our cell decomposition. Now consider the map \[f:\,\mbox{S}^{15} \longrightarrow \mathbb{O}\textup{P}^1,\quad (x,y) \mapsto [x,y]\,,\]
	where we think of $\mbox{S}^{15}$ as a subset of $\mathbb{R}^{16} = \mathbb{R}^{8} \times \mathbb{R}^{8}$.

	\begin{lemma}
		We have $\mathbb{O}\textup{P}^2 =  \mathbb{O}\textup{P}^1 \cup_f \mbox{D}^{16}$.
	\end{lemma}
	
	\begin{proof}
		A map from $\mbox{D}^{16}$ to $\mathbb{O}\textup{P}^2$ coinciding with $f$ on $\partial \mbox{D}^{16}$ is given by \[(x,y) \mapsto \left[x,y, \sqrt{1-\left\Vert x\right\Vert^2-\left\Vert y\right\Vert^2}\right]\,.\] It is easily checked that this map induces a bijective continuous map from $\mathbb{O}\textup{P}^1 \cup_f \mbox{D}^{16}$ to $\mathbb{O}\textup{P}^2$ which is thus a homeomorphism since $ \mathbb{O}\textup{P}^1 \cup_f \mbox{D}^{16}$ is quasi-compact and $\mathbb{O}\textup{P}^2$ is Hausdorff.
	\end{proof}
	
	\section{Cohomology of $\mathbb{O}\textup{P}^2$}
	
	After having constructed this very simple cell structure for $\mathbb{O}\textup{P}^2$, it is easy to compute the cohomology via the cellular cochain complex: 
	
	\begin{corollary}\label{cohomology}
		Let $A$ be any abelian group. Then \[\textup{H}^k(\mathbb{O}\textup{P}^2, A) \cong \begin{cases} A\,, \quad k=0, 8\; \textup{or}\; 16\,, \\ 0\,,\quad \textup{else.} \end{cases}\]
	\end{corollary}
	
	The homology groups are computed in exactly the same way and with the same result. 
	
	Similarly, we can compute the homotopy groups of $\mathbb{O}\textup{P}^2$: By cellular approximation, $\pi_n(\mathbb{O}\textup{P}^2,*)=0$ for $n\le 7$ and $\pi_n(\mathbb{O}\textup{P}^2,*)=\pi_n(\mbox{S}^8)$ for $n\le 14$.
	
	The following question has been brought to the author's attention by Jens Reinhold.
	
	\begin{question}
		Is there a closed, $8$-connected manifold of positive dimension with odd Euler characteristic?
	\end{question}
	
	The octonionic projective plane gives such a manifold which is $7$-connected. An $8$-connected example would have to have dimension divisible by $32$, by a result of Hoekzema \cite[Thm.~1.2, Cor.~4.2]{hoekzema}.
	
	\subsection{Connection with the Hopf invariant $1$ problem} 
	
	The Hopf invariant is a classical invariant for maps $f\colon \mathrm{S}^{2n-1}\rightarrow \mathrm{S}^n$, with $n>1$. It goes back to work of Hopf in the 1930's. We quickly recall its definition from \cite{mosher-tangora}. Let us consider the mapping cylinder of such a map $f$. By inspection of the cellular cochain complex, as in Corollary \ref{cohomology} above, it has cohomology groups in degree $n$ and $2n$ which are cyclic with generators $\tau$ and $\sigma$. These are unique up to sign, depending on the orientation of the two spheres. The Hopf invariant $H(f)$ is defined by the formula \[\tau^2=H(f)\cdot \sigma\,.\] It is unique up to sign, which only depends on the orientation of $\mathrm{S}^{2n-1}$ since $\tau$ appears squared.
	
	The question in which dimensions there exists a map of Hopf invariant $1$ was a famous open problem in the early days of algebraic topology, until Adams proved in 1960 that this is only the case for $d\in \{1,2,4,8\}$. The sought maps for $d=2,4,8$ can be constructed as the attaching maps of the top dimensional cell in the projective planes over the complex numbers, quaternionics and octonionics, and we will now prove this for the octonionics, by analysing the ring structure on the cohomology of $\mathbb{O}\textup{P}^2$.
	
	\begin{thm}
		The attaching map $f\colon \mbox{S}^{15} \longrightarrow \mbox{S}^8$ of the $16$-cell in $\mathbb{O}\textup{P}^2$ has Hopf invariant $\pm 1$, the sign depending on the orientation of $\mbox{S}^{15}$.
	\end{thm}
	
	\begin{proof}
		Let $\tau$ and $\sigma$ be generators of $\mbox{H}^{8}(\mathbb{O}\textup{P}^2, \mathbb{Z})$ and $\mbox{H}^{16}(\mathbb{O}\textup{P}^2, \mathbb{Z})$, respectively. Since $\mathbb{O}\textup{P}^2$ is the mapping cylinder of $f$, we just have to show that $\tau^2 = \sigma$ (up to sign) by the definition above. 
		Since we have seen $\mathbb{O}\textup{P}^2$ to be a closed manifold which is orientable since it is simply-connected, we can profit of Poincar\'e duality to do so: Let  $\mu \in \mbox{H}_{16}(\mathbb{O}\textup{P}^2, \mathbb{Z})$ be a fundamental class. Using the universal coefficient theorem  as well as Poincar\'e duality, we get isomorphisms
		\[\mbox{H}^n(\mathbb{O}\textup{P}^2, \mathbb{Z}) \cong \mbox{Hom}(\mbox{H}_n(\mathbb{O}\textup{P}^2, \mathbb{Z}), \mathbb{Z})\cong  \mbox{Hom}(\mbox{H}^{16-n}(\mathbb{O}\textup{P}^2, \mathbb{Z}), \mathbb{Z}) \,,\]
		where the map from the left to the right maps $f$ to the linear map \[g\mapsto \langle f, \mu \cap g \rangle = \langle g\cup f, \mu\rangle\,. \]
		Now, $\mbox{Hom}(\mbox{H}^8(\mathbb{O}\textup{P}^2, \mathbb{Z}), \mathbb{Z})$ is isomorphic to $\mathbb{Z}$, generated by the two isomorphisms. Thus, $\tau$ has to be mapped to an isomorphism $\mbox{H}^8(\mathbb{O}\textup{P}^2, \mathbb{Z})\longrightarrow \mathbb{Z}$, which in turn has to map $\tau$ to a generator of $\mathbb{Z}$, so \[\langle \tau^2, \mu \rangle = \pm 1\,.\]
		Now, setting $\tau^2= k\sigma$ with $k\in\mathbb{Z}$, we get \[k\cdot \langle\sigma, \mu\rangle = \pm 1\,,\] so $k$ divides $1$, giving $k=\pm 1$ and thus $\tau^2=\sigma$.
	\end{proof}
	
	Note that the constructions we have carried out in the last two sections can be done in a much more general setting: Suppose that $\mathbb{A}$ is a normed real algebra $\mathbb{A}$ of dimension $d<\infty$ with conjugation, which is alternative and has no zero divisors. Then we can write down the same formulas as above to define a topological space $\mathbb{A}\textup{P}^2$, prove that it is a manifold and give it a cell structure with one cell in each of the dimensions $0$, $d$ and $2d$. The attaching map of the $2d$-cell will then be a map $\mbox{S}^{2d-1} \longrightarrow \mbox{S}^d$ of Hopf invariant $1$. 
	
	In \cite[Sec.~8.1, 8.2, 9.1]{numbers}, it is shown that the existence of the two extra structures on $\mathbb{A}$ does not have to be claimed on its own: For any real alternative division algebra, there is a canonical norm and conjugation\footnote{As the alert reader may have noticed, we have used exactly one more property of the conjugation, namely the fact that $z^*$ always lies in the subalgebra generated by $z$. However, the canonical conjugation constructed in \cite{numbers} always has this property.}. 
	
	Summarising, we have argued that the following holds: 
	
	\begin{thm}
		If there exists a real alternative division algebra of dimension $d$, then there is a map {$\mbox{S}^{2d-1} \longrightarrow \mbox{S}^d$} of Hopf invariant $1$.
	\end{thm}
	
	By Adams' result, this is only the case for $d\in \{1,2,4,8\}$. Thus the construction of the projective plane shows that a real alternative division algebra can only exist in dimensions $1,2,4$ and $8$. Of course, this is still true if the alternativity claim is dropped, but one needs a different proof for this \cite[Sec.~2.3]{hatvec}. 
	
	\subsection{Non-existence of higher octonionic projective spaces}

	As pointed out above, our construction of the octonionic projective space $\mathbb{O}\textup{P}^2$ doesn't generalise to higher dimensional projective spaces since we have intensively used the fact that all calculations are done in associative subalgebras of $\mathbb{O}$. However, there is also a conceptual reason that there \emph{can't} be a space which deserves to be called $\mathbb{O}\textup{P}^3$ (or $\mathbb{O}\textup{P}^n$ for some $n\ge 3$) which we will now explain. 
	
	We will only claim two properties of our wannabe projective octonionic $3$-space: it should be a closed manifold, and it should have a cell structure with $\mathbb{O}\textup{P}^2$ as the $16$-skeleton and only one more $24$-cell. It then follows directly that the cohomology $\mbox{H}^*(\mathbb{O}\textup{P}^3, \mathbb{Z})$ is $\mathbb{Z}$ in dimensions $0, 8, 16$ and $24$ and trivial otherwise.
	
	By a similar argument as for $\mathbb{O}\textup{P}^2$, we also get the ring structure on the cohomology: \[\mbox{H}^*(\mathbb{O}\textup{P}^3, \mathbb{Z}) \cong \mathbb{Z}[x]/(x^4), \quad |x| = 8.\]
	To see this, note that the inclusion of the $16$-skeleton induces an isomorphism on cohomology in degrees smaller than $23$ which respects the multiplicative structure, thus it is sufficient to show that a fourth power of the generator of $\mbox {H}^8$ generates $\mbox{H}^{24}$. But this is done in a very similar way as for the octonionic projective plane, using Poincar\'e duality. 
	
	However, using Steenrod powers modulo $2$ and $3$, one can show that a space with cohomology $\mathbb{Z}[x]/(x^m)$, $m>3$, can only exist if $x$ has degree $2$ or $4$ \cite[Sec.~4.L]{hat}. 
	
	\subsection{The story continues}
	
	We have used the octonions to construct a map between spheres of Hopf invariant $1$. There are other phenomena in the intersection of algebra, topology and geometry that show deep relations with the octonions. Examples include exotic spheres, Bott periodicity and exceptional Lie groups (these can be used to see that $\mathbb{O}\textup{P}^2$ is a homogeneous space, for instance). The article \cite{baez} explains these and many more interesting examples.

	\bibliographystyle{alpha} 
	\bibliography{octonions} 
	
\end{document}